\documentclass[reqno,12pt]{amsart}

\usepackage{multirow}
\usepackage[centering,text={15.5cm,22cm},
		marginparwidth=20mm]{geometry}

\usepackage{amsmath}
\usepackage{amssymb}
\usepackage{graphicx}
\usepackage{fontenc}
\usepackage{amsthm}
\usepackage{enumerate}
\usepackage{bbold}
\usepackage{color}
\theoremstyle{definition}
\newtheorem{theorem}{Theorem}
\newtheorem*{theorema}{Theorem}

\newtheorem{lemma}[theorem]{Lemma}
\newtheorem{proposition}[theorem]{Proposition}

\newtheorem{definition}[theorem]{Definition}
\newtheorem{example}[theorem]{Example}
\newtheorem{remark}[theorem]{Remark}

\numberwithin{equation}{section}
\newcommand{\Mbar}[2]{\overline{\mathcal{M}}_{#1,#2}}

\newcommand{\A}{\mathbf{A}}
\newcommand{\B}{\mathbf{B}}
\newcommand{\K}{\lambda}
\newcommand{\p}{\mathrm{\bf{p}}}
\newcommand{\q}{\mathrm{\bf{q}}}
\newcommand{\tq}{\mathrm{\tilde{\bf{q}}}}
\newcommand{\br}{\mathrm{\bf{r}}}
\newcommand{\tbr}{\mathrm{\tilde{\bf{r}}}}
\newcommand{\s}{\mathrm{\bf{s}}}
\newcommand{\T}{\mathrm{\bf{t}}}
\newcommand{\Tk}{\mathrm{\bf{T}_k}}

\newcommand{\PP}{\mathrm{P}}

\newcommand{\Set}[1]{\left\lbrace #1 \right\rbrace}

\DeclareRobustCommand{\stirling}{\genfrac\{\}{0pt}{}}
\DeclareRobustCommand{\Choose}{\genfrac{(}{)}{0pt}{}}

\everymath{\displaystyle}

\begin{document}
\title[]{The product rule in $\kappa^*(\mathcal{M}_{g,n}^{ct})$ }%
\author{I. Setayesh}%
\address{Department of Mathematics, Faculty of Mathematical Sciences, Tarbiat Modares University, P.O. Box 14115-137, Tehran, Iran.
}
\address{School of Mathematics, Institute for Research in Fundamental Sciences (IPM), P.O. Box 19395-5746, Tehran, Iran.}
\email{setayesh@ipm.ir}%

\thanks{}%


\begin{abstract}
We describe explicit formulas for the product rule in $\kappa^*(\mathcal{M}_{g,n}^{ct})$. 
\end{abstract}
\maketitle

\section{Introduction}

Let $\epsilon : \Mbar{g}{n+1} \to \Mbar{g}{n}$ be the forgetful map, viewed as the universal curve over $\Mbar{g}{n}$, the moduli space of stable curves of genus $g$ with $n$ marked points. Let $\mathbb{L}_{i} \to \Mbar{g}{n+1}$ be the cotangent line bundle over $\Mbar{g}{n+1}$, whose fiber over a given curve is the cotangent space at the $i^{th}$ marked point. Define
$$\psi_i= c_1(\mathbb{L}_{i})\in A^{1}(\overline{\mathcal{M}}_{g,n+1})\ \ 
\text{and}\ \ \kappa_i = \epsilon_{*}(\psi_{n+1}^{i+1}) 
\in A^{i}(\overline{\mathcal{M}}_{g,n}) \ .$$

By restriction, one can define the psi and kappa classes over $\mathcal{M}_{g,n}$, the moduli space of smooth curves of genus $g$ with $n$ marked poines, and $\mathcal{M}_{g,n}^{ct}$, the moduli space of curves of compact type. Consider the sub-ring of $A^*(\mathcal{M}_{g,n}^{ct})$ generated by the kappa classes, and call it the kappa ring. In \cite{Rahul-k} Pandharipande described an additive basis for the kappa ring.

\begin{theorema}\cite{Rahul-k}
Given $D,n \in \mathbb{N}$, the set
$$\{ \kappa_{\mathbf{p}} \ | \ \mathbf{p} \in P(D,2g+n-D-2)\ \} \ $$
generates $\kappa^D(\mathcal{M}_{g,n}^{ct})$ as a $\mathbb{Q}$-vector space, and for $n>0$ it is a $\mathbb{Q}$-basis.
\end{theorema}

The natural question to ask, as first raised by Pandharipande \cite{Rahul-k}, is to determine explicit formulas for the product rule in the kappa ring of $\mathcal{M}_{g,n}^{ct}$, that is the main result of this paper.

\begin{theorem}\label{main}
Let $\A=\Set{a_1,\cdots,a_k}$ be a multi-set of integers, and $n,g,d \in \mathbb{N}$ be integers such that $d=2g+n-\sum a_i-2$. In $A^{*}(\mathcal{M}_{g,n}^{ct})$ we have:
$$\kappa_{a_1}\cdots\kappa_{a_k}=\sum_{\p\in SP(\A,d)} x_{\p} \kappa_{\p(\A)}$$
where 
$$x_{\p}=\sum_{\T \leq \br \leq \p \in SP(\A)}  \frac{(-1)^{\ell(\A)+\ell(\T)+\ell(\br)+M}}{(|\T|+\mathbb{1})!} 
\Choose{\ell(\T)-\ell(\br)}{M-\ell(\br)}
\prod_{j=1}^{\ell(\br)} (|\br_j+\mathbb{1}|)! \prod_{l=1}^{\ell(\p)}(\ell(\br|_{\p_l})-1)!  $$ in which $M = \min \{ \ell(\T),d \}$, $SP(\A)$ (resp. $SP(\A,d)$) is the set of partitions of the multi-set $\A$ (resp. into at most $d$ parts) and $\ell(\p)$ is the number of components of $\p$ (for notations see Definition \ref{def1} and Definition \ref{def2}). 
\end{theorem}

{\bf{Plan of the paper.}}
In Section 2 we review some known results relating kappa classes and pushforwards of the psi classes. Section 3 contains the proof of the main theorem. The main idea is to use the following theorem of Pandharipande.

\begin{theorem}\cite{Rahul-k}\label{iota}
There is a canonical surjective map $\iota_{g,n}:\kappa^*(\Mbar{0}{n+2g}) \to \kappa^*(\mathcal{M}_{g,n}^{ct})$ which is an isomorphism for $n>0$.
\end{theorem}

 This allows us to reduce the computation to the case of genus zero. In the genus zero case, by the work of Keel \cite{Keel} we have a very good understanding of the Chow ring, and we can explicitly compute all the required classes. In Section 4 we prove the combinatorial identity used in the Section 3.

{\bf{Acknowledgement.}} I would like to thank M. Einollah Zade, M. Saghafian and E. Salavati for their collaboration in the proof of Theorem \ref{combthm}. I am also grateful to E. Eftekhary and F. Janda for helpful comments on an earlier version of this paper. This work was partially done when the author visited the Institute for Mathematical Research (FIM) in Zurich and the result was first presented in a lecture in the Einstein series in Algebraic Geometry at Humboldt University supported by the Einstein Stiftung in Berlin.

\section{Kappa and Psi classes}
In this section we explain the relation between the kappa classes and pushforwards of the psi classes. Before stating the results, we need to fix some notations.

\begin{definition}
Let $\pi_{g,n}^{k}:\Mbar{g}{n+k}\to \Mbar{g}{n}$ denote the forgetful map which forgets the last $k$ marked points, and let $\q$ be the multi-set $\Set{q_1,\cdots ,q_k}$. We define:
\begin{itemize}
\item $\psi(\q)=\psi(q_1,\cdots,q_k):=\left(\pi_{g,n}^{k}\right)_*
\left(\prod_{i=1}^k\psi_{n+i}^{q_i+1}\right)$.
\item $\kappa_{\q} = \kappa_{q_1,\cdots,q_k}:=\prod_{i=1}^r\kappa_{q_i} .$
\end{itemize}

\end{definition}

\begin{definition}
Let $\A=\Set{a_1,\cdots,a_k}$ be a multi-set. Given $\sigma\in S_k$ with the cycle decomposition $\sigma = \gamma_1 \ldots \gamma_r$ (including the 1-cycles). We define:
\begin{itemize}
\item $\sigma_i(\mathbf{a}):=\sum_{j\in \gamma_i}a_j $.
\item $\sigma(\A) = \Set{\sigma_1(\A),\cdots,\sigma_r(\A)}$ (as multi-set).
\end{itemize}

\end{definition}

\begin{definition}\label{def1}
Let $\A=\Set{a_1,\cdots,a_k}$ be a multi-set, we denote the set of partitions of $\A$ by $SP(\A)$. Given $\p \in SP(\A)$, we use the following notations.
\begin{itemize}
\item $|\A|=\sum_{j=1}^k a_j$.
\item $SP(\A;l)$ (resp. $SP(\A,l)$) denotes the set of partitions of $\A$ with exactly (resp. at most) $l$ parts.
\item We denote the components of $\p$ by $\p_i$, i.e. $\p=\displaystyle\bigsqcup \p_i$.
\item $\ell(\p)$ denotes the number of components of $\p$. Also we write $\ell(\A)=k$ and from the context it should be clear whether we work with a multi-set or a partition, so there should be no confusions.
\item $|\p| = \Set{|\p_1|,\cdots,|\p_{\ell(\p)}|}$ (as multi-set)
\item By abuse of notation we write $\psi(\p)$ instead of $\psi(|\p|)$, i.e.
$$\psi(\p)= \psi(|\p|)= \left(\pi_{g,n}^{\ell(\p)}\right)_*
\left(\prod_{i=1}^{\ell(\p)}\psi_{n+i}^{|\p_i|+1}\right)  .$$
\end{itemize}
\end{definition}

The relation between the kappa classes and pushforward of the psi classes, due to Faber, is (see \cite{AC}):
$$ \psi(a_1,...,a_{k}) = \sum_{\sigma\in S_k} \kappa_{\sigma(\mathbf{a})}\ .$$

\begin{lemma}(\cite{AC}, \cite[Lemma 11]{Rahul-k})
The subset of $A^d(\Mbar{g}{n})$ defined by
$$\Set{\psi(\A)\ \big|\ \A\in\PP(d)}\ \ \text{and} \ \
\Set{\kappa_{\A}\ \big|\ \A\in\PP(d)}$$
are related by an invertible linear transformation independent of $g$ and $n$.
\end{lemma}

The map in one direction is given by the Faber's formula. The inverse map is given by the following proposition.

\begin{proposition}\label{inv-kappa}
$\kappa_{\A}=\sum_{\p\in SP(\A)}(-1)^{n+\ell(\p)}\psi(\p).$
\end{proposition}

\begin{proof}

Given $\q\in SP(\A)$, we say that a pair $\left(\p,\sigma\right)$ with $\p\in SP(\A)$ and $\sigma \in S_{\ell(\p)}$ splits $\q$ if $\sigma(|\p|)=\q$. Using Faber's formula we have:

$$\begin{array}{rl}
\sum_{\p\in SP(\A)}(-1)^{n+\ell(\p)}\psi(\p) &= \sum_{\p\in SP(\A)}\sum_{\sigma \in S_{\ell(\p)}}(-1)^{n+\ell(\p)}\kappa_{\sigma(|\p|)} \\
&=(-1)^n \sum_{\q\in SP(\A)} \kappa_{\q} \left( \sum_{(\p,\sigma) \text{ splits }\q}(-1)^{\ell(\p)}\right) .
\end{array}$$

Let $\stirling{n}{k}$ denote the number of partitions of a set of $n$ elements into $k$ parts (the Stirling numbers of the second kind). Given a splitting $\left(\p,\sigma\right)$ of $\q=\amalg_{i=1}^{r}\q_i$, by restriction of $\p$ to each $\q_i$ we obtain an element $\p_i$ of $SP(\q_i)$ and the restriction of $\sigma$ is a permutation with only one cycle of length $\ell(\p_i)$. Using the identity 
$$\sum_{k=1}^{n} (-1)^{k}(k-1)! \stirling{n}{k} =-\delta_1^n$$
\noindent for the  Stirling numbers of the second kind (see \cite{Vanlint}), we have:
$$\begin{array}{rl}
\sum_{(\p,\sigma) \text{ splitting of }\q}(-1)^{\ell(\p)}
&=\prod_{i=1}^{r}\left(\sum_{\p_i \in SP(\q_i)} (-1)^{\ell(\p_i)}(\ell(\p_i)-1)! \right) \\
&=\prod_{i=1}^{r} \left(\sum_{k=1}^{\ell(\q_i)} (-1)^{k}(k-1)! \stirling{\ell(\q_i)}{k} \right) \\
&=(-1)^r \prod_{i=1}^{r} \delta_1^{\ell(\q_i)}.
\end{array}$$

 Hence if for some $i$ we have $\ell(\q_i)>1$ then the coefficient of $\kappa_{\q}$ is zero. Therefore the only term with non-zero contribution is $\q_0=\Set{a_1}\cup\Set{a_2}\cup\cdots\cup\Set{a_n}$. Therefore we have:

 $$\begin{array}{rl}
\sum_{\p\in SP(\A)}(-1)^{n+\ell(\p)}\psi(\p) &= \kappa_{\q_0} \\
&= \kappa(\A)
\end{array}$$

\end{proof}

\begin{example}
\begin{footnotesize}
$$\begin{array}{rl}
\sum_{\p\in SP(\Set{a,b,c})}(-1)^{3+\ell(\p)}\psi(\p)
=&
\psi(a,b,c)
-\psi(a+b,c)-\psi(a+c,b)-\psi(b+c,a)+\psi(a+b+c)
\\
=&
(\kappa_a \kappa_b \kappa_c + \kappa_{a+b}\kappa_{c}+ \kappa_{a+c}\kappa_{b}+\kappa_{b+c}\kappa_{a}+ 2 \kappa_{a+b+c}) 
\\
&
- (\kappa_{a+b}\kappa_{c}+  \kappa_{a+b+c})
- (\kappa_{a+c}\kappa_{b}+  \kappa_{a+b+c})
- (\kappa_{b+c}\kappa_{a}+  \kappa_{a+b+c})
\\
&
+  \kappa_{a+b+c}
\\
=& \kappa_{a+b+c}
\end{array}$$
\end{footnotesize}

\end{example}

\section{Product of kappa classes}
 As we explained in the introduction, Theorem \ref{iota} shows that any relation in $\kappa^*(\Mbar{0}{n})$ is also valid in $\kappa^*(\mathcal{M}_{g,n}^{ct})$. Moreover for $n>1$ all the relations in $\kappa^*(\mathcal{M}_{g,n}^{ct})$ are obtained in this manner. Therefore in order to prove Theorem \ref{main}, it is enough to show that those relations hold in $\kappa^*(\Mbar{0}{n})$. Thus we start by computing the product rule in the kappa ring of $\Mbar{0}{n}$, but we need to fix some more notations.

\begin{definition}\label{def2}
Let $\A=\Set{a_1,\cdots,a_k}$ be a multi-set. 
\begin{itemize}
\item $\A!=\displaystyle\prod_{i=1}^k a_i!$.
\item $\A+\mathbb{1}=\Set{a_i+1\mid i=1,\cdots,k}$.
\item $\Choose{|\A|}{\A}=\Choose{a_1+a_2+\cdots+a_k}{a_1,a_2,\cdots,a_k}$.
\item Given $\p,\q \in SP(\A)$, we write $\q \leq \p$  if $\q$ is a refinement of $\p$. We use the following notations.
\begin{itemize}
\item $[\p:\q]\in SP(|\q|)$ denotes the partition induced from $\p$ on the multi-set $|\q|$. Note that $\ell(\p)=\ell([\p:\q])$.
\item $\q|_{\p_i}\in SP(\p_i)$ denotes the partition induced from $\q$ on the multi-set $\p_i$ (by restriction).
\end{itemize}
\item $\K_{\A}=\K_{a_1,\cdots,a_n}:=\sum_{\p\in SP(\A)}(-1)^{n+\ell(\p)}  {\Choose{|(|\p|+\mathbb{1})|}{|\p|+\mathbb{1}}} . $
\end{itemize}

\end{definition}

In the following proposition we compute the product rule in the top degree of $\kappa^*(\Mbar{0}{n})$.

\begin{proposition}\label{baby-case}
Let $\A=\Set{a_1,\cdots,a_k}$ be a multi-set of integers such that $\sum_i a_i=n-3$.
Then in  $A^{n-3}(\Mbar{0}{n})$ we have:
$$\kappa_{a_1,\cdots,a_l}=\K_{a_1,\cdots,a_l}\kappa_{n-3} . $$

\end{proposition}

\begin{proof}

Since $A^{n-3}(\Mbar{0}{n})=\mathbb{Q}$, we have to integrate both sides over $\Mbar{0}{n}$. Given $\p=\amalg_{i=1}^r \p_i \in SP(\A)$, we have:

 $$\begin{array}{rl}
\int_{\Mbar{0}{n}} \psi(\p) &= \int_{\Mbar{0}{n+r}} \prod_{i=1}^r\psi_{n+i}^{|\p_i|+1} \\
&=  {\Choose{|\p_1|+\cdots+|\p_r|+r}{|\p_1|+1,\cdots,|\p_r|+1}} \\
&=  {\Choose{|(|\p|+\mathbb{1})|}{|\p|+\mathbb{1}}}  .
\end{array}$$

Therefore the result follows from Proposition~\ref{inv-kappa}.
\end{proof}

%

\begin{definition}
Let $\A=\Set{a_1,\cdots,a_k}$ be a multi-set, $\q\leq\p\in SP(\A)$. 
\begin{eqnarray*}
N_{\A}=N_{a_1,\cdots,a_n}&:=&\sum_{\br \in SP(\A)} (-1)^{n+\ell(\br)} (\ell(\br)-1)!  \frac{\prod_{j=1}^{\ell(\br)}(|\br_j+\mathbb{1}|)!}{\prod (a_i+1)!} 
\\
&=&\sum_{\br \in SP(\A)} (-1)^{n+\ell(\br)}(\ell(\br)-1)!  \prod_{j=1}^{\ell(\br)}\Choose{|\br_j+\mathbb{1}|}{\br_j+\mathbb{1}} .
\\
N_{\p}&:=& \prod_{i=1}^{\ell(\p)} N_{\p_i} .
\end{eqnarray*}

Note that $N_{a_1}=1$. 
\end{definition}

\begin{definition}
A genus zero $n-$pointed \emph{stable weighted graph} is a connected graph $G$ together with a map $m:V(G)\to 2^{\Set{1,\cdots,n}}$ that satisfies the following conditions.
\begin{itemize}
\item The sets $m(v)$ (for $v\in V(G)$) form a partition of $\Set{1,\cdots,n}$. We call $m(v)$ the set of markings at the vertex $v$.
\item For each $v\in V(G)$ we have $d(v):=deg(v)+\sharp m(v) \geq 3$.
\item $H^1(G)=0$, i.e. $G$ is a tree.
\end{itemize}

For any genus zero $n-$pointed stable weighted graph $G$, we have a cycle $[G]\subset \Mbar{0}{n}$ obtained as follows. We take $[G]$ to be the image of the map $$\prod_{v\in V(G)} \Mbar{0}{d(v)} \to \Mbar{0}{n}$$ obtained by gluing along the edges of $G$. 
\end{definition}

\begin{remark}\label{Keel-remark}
In \cite{Keel} Keel proved that $A^{*}(\Mbar{0}{n})$ is generated by cycles of the form $[G]$, and it has perfect pairing. Therefore given a Chow class $\alpha\in A^{*}(\Mbar{0}{n})$, in order to show that $\alpha$ vanishes it is enough to show that the pairing of $\alpha$ against all the classes of the form $[G]$ are zero. Note that if $\alpha$ belongs to the kappa ring, then $\alpha.[G]$ depends only on the dimension of components of $[G]$ and not the distribution of the markings. Thus $\alpha.[G]$ depends only on the dimension sequence of $G$, i.e. $\Set{d(v)-3:v\in V(G)}$. 
\end{remark}

The following theorem describes an additive basis for the kappa ring of $\Mbar{0}{n}$.

\begin{theorem}\cite{Rahul-k}\label{Rahul-additive}
Given $D\in \mathbb{N}$, a $\mathbb{Q}$-basis of $\kappa^D(\Mbar{0}{n})$ is given by
$$\{ \kappa_{\mathbf{p}} \ | \ \mathbf{p} \in P(D,n-D-2)\ \} \  .$$
\end{theorem}

\begin{theorem}
Let $\A=\Set{a_1,\cdots,a_k}$ be a multi-set of integers, and $n,d\in \mathbb{N}$ be such that $d=n-\sum a_i-2$. In $A^{*}(\Mbar{0}{n})$ we have:
$$\kappa_{a_1,\cdots,a_k}= \sum_{\p\in SP(\A,d)} \left( \sum_{\substack{\q \leq \p \\ \ell(\q)\leq d} } \K_{q} N_{[\p:\q]} \right) \kappa_{\p}.$$

\end{theorem}

\begin{proof}
Using Remark \ref{Keel-remark} it is enough to check that the integral of both sides against cycles of the form $[G]$ are equal. Note that both sides belong to $A^{n-d-2}(\Mbar{0}{n})$, so they can be paired with cycles with exactly $(n-3)-(n-d-2)=d-1$ nodes. Thus we have to consider cycles that their dimension sequence have exactly $d$ terms.

 If the integral of the left (or right) hand side over $[G]$ is non-zero, then there is a way to put $\kappa_{a_1},\cdots,\kappa_{a_l}$ on irreducible components of $[G]$ such that the dimension of each component is equal to the sum of degree of kappa classes over it. Therefore there is a partition $\p$ of the multi-set $\{a_1,\cdots,a_l\}$ (with at most $d$ parts) such that the multi-set of dimension of irreducible components of $[G]$ is equal to $|\p|$. We call such $G$ a $\p-$graph, and call $[G]$ a $\p-$cycle. Therefore it is enough to check our claim for all the strata of the form $[G]$, where $G$ is a $\p-$graph for some $\p\in SP(\A,d)$.
 
By Theorem \ref{Rahul-additive}, kappa classes of the from $\kappa_{\p}$ for $\p\in SP(\A,d)$ form an additive basis for the kappa ring, therefore for each $\p$ there exist (a unique) $x_{\p}$ such that 
\begin{equation}\label{coefficients}
\kappa_{a_1,\cdots,a_l}=\sum_{\p\in SP(\A,d)} x_{\p} \kappa_{\p(\A)} .
\end{equation}
Hence we have to prove that
\begin{equation}\label{xp}
x_{\p}=\sum_{\substack{\q \leq \p \\ \ell(\q)\leq d} } \K_{q} N_{[\p:\q]} .
\end{equation}

We prove this by induction on $\ell(\p)$. Given a partition $\p$ with $\ell(\p)=d$, we integrate both sides of (\ref{coefficients}) against a $\p-$cycle $[G]$. By Proposition \ref{baby-case} the integral of the left side is $\K_{\p}$, and the integral of the right hand side is $x_{\p}$. Hence we get $x_{\p}=\K_{\p}$, which confirms \ref{xp}. 

Fix $\p \in SP(\A)$ with $\ell(\p)=e<d$. Similarly by Proposition \ref{baby-case} the integral of the left hand side of (\ref{coefficients}) is $\prod\K_{\p_i}=\K_{\p}$, and the integral of the right hand side is given by $$\sum_{\substack{\q \leq \p \\ \ell(\q)\leq d}} x_{\q}\K_{[\p:\q]} .$$ 
Since any partition $\q$ in this sum, except $\p$, has length at least $e+1$ by the induction hypothesis we know that $$x_{\q}=\sum_{\substack{\br \leq \q \\ \ell(\br)\leq d} } \K_{\br} N_{[\q:\br ]} .$$
Thus we obtain 
\begin{eqnarray*}
x_{\p} &=& \K_{\p} - \sum_{\substack{\q < \p \\  \ell(\q)\leq d}} \left\lbrace\sum_{\substack{\br \leq \q \\ \ell(\br)\leq d} } \K_{\br} N_{[\q:\br]} \right\rbrace \K_{[\p:\q]}
\\
&=& \K_{\p} - \sum_{\substack{\br < \p \\ \ell(\br)\leq d}} \K_{\br} \left\lbrace\sum_{\br \leq \q < \p }  N_{[\q:\br ]}\K_{[\p:\q]} \right\rbrace .
\end{eqnarray*}

Let $\p=\p_1\amalg\cdots\amalg\p_e$, in order to simplify the formulas we denote $\q|_{\p_i}$ (resp. $\br|_{\p_i}$) by $\q'_i$ (resp. $\br'_i$), then $\K_{[\p:\q]}=\prod_{i=1}^{e} \K_{|\q'_i|}$ and $N_{[\q:\br ]}=\prod  N_{[\q'_i:\br'_i ]}$. Therefore we obtain:

\begin{eqnarray*}
x_{\p} &=& \K_{\p} - \sum_{\substack{\br < \p \\ \ell(\br)\leq d}} \K_{\br} \left\lbrace\sum_{\br \leq \q < \p } \prod_{i=1}^{e}  N_{[\q'_i:\br'_i ]}\K_{|\q'_i|} \right\rbrace \\
&=&
\K_{\p} - \sum_{\substack{\br < \p \\  \ell(\br)\leq d}} \K_{\br} \left[-N_{[\p:\br]}+\prod_{i=1}^{e}\left( \sum_{\substack{\br'_i \leq \q'_i \\ \q'_i \in SP(\p_i) }}  N_{[\q'_i:\br'_i ]}\K_{|\q'_i|} \right) \right] .
\end{eqnarray*}
In the first line $\q < \p$ is a proper sub partition, and in the second identity we added the term $-N_{[\p:\br]}$ and took the sum over all partitions. Therefore in order to finish the proof, it is enough to check that the terms in the parenthesis are zero, which is the content of the following lemma.
\end{proof}

\begin{lemma}\label{NK-lemma}
Given a multi-set $\A$ and a partition $\s$ of $\A$. We have $$\sum_{\s \leq \p \in SP(\A)}  N_{[\p: \s]}\K_{|\p|}=0 .$$ 
\end{lemma}

\begin{proof}
We have $$\sum_{\s \leq \p \in SP(\A)}  N_{[\p: \s]}\K_{|\p|}= \sum_{\p \in SP(|\s|)}  N_{\p}\K_{|\p|} .$$
Hence it is enough to check that the term in the left is zero. We denote the multi-set $\s$ by $\B$, and $b:=\sharp\B$. 

\begin{eqnarray*}
\sum_{\p  \in SP(\B)}  N_{\p}\K_{|\p|} & = & \sum_{\p \in SP(\B)} \K_{|\p|} \prod_{i=1}^{\ell(\p)} N_{\p_i}
\\
& = & \sum_{\br\leq \p \leq \q \in SP(\B)}  (-1)^{\ell(\p)+\ell(\q)} \Choose{|\q +\mathbb{1}|}{\q +\mathbb{1}} \cdot
\\
& & 
\hspace{10mm}  (-1)^{b+\ell(\br)} \prod_{i=1}^{\ell(\p)} \prod_{j=1}^{\ell(\br|_{\p_i})} \Choose{|\br_i^j +\mathbb{1}|}{\br_i^j +\mathbb{1}} (\ell(\br|_{\p_i})-1)!
\\
& &  ( \text{where } \br|_{\p_i}=\coprod_{j=1}^{\ell(\br|_{\p_i})}\br_i^j \text{ is the partition induced by }\br \text{ on the set } \p_i )
\end{eqnarray*}

The data of a triple $\br\leq \p \leq \q \in SP(\B)$ is equivalent to the data of $\br\leq  \q \in SP(\B)$ plus the data of a partition $\p':=[\p:\br] \leq [\q:\br] \in SP(|\br|)$. The data of a partition $\p':=[\p:\br] \leq [\q:\br] \in SP(|\br|)$ is equivalent to the following. 

For each $1 \leq l \leq \ell(\q)=\ell([\q:\br])$, the data of a partition $$\p'_l=\coprod_{j=1}^{\ell(\p'_l)}{\p'_l}^j$$ of the multi-set $[\q:\br]_l$. 

Note that any multi-set $\p_i$ induces a unique ${\p'_l}^j$ and we have 

$$\begin{array}{rl}
\ell(\br|_{\p_i})&= \text{ number of components or }\br \text{ in } \p_i \\
&= \text{ number of elements of }{\p'_l}^j \text{ as a partition of the multi-set }|\br| \\
&= \ell({\p'_l}^j)
\end{array}$$
Also if we vary $i$ and $j$, $\br_i^j$ runs over all components of $\br$, hence
$$ \prod_{i=1}^{\ell(\p)} \prod_{j=1}^{\ell(\br|_{\p_i})} \Choose{|\br_i^j +\mathbb{1}|}{\br_i^j +\mathbb{1}}= \prod_{i=1}^{\ell(\br|)} \Choose{|\br_i +\mathbb{1}|}{\br_i +\mathbb{1}} . $$

Therefore

\begin{eqnarray*}
\sum_{\p  \in SP(\B)}  N_{\p}\K_{|\p|}
& = & \sum_{\br\leq \q \in SP(\B)}  (-1)^{\ell(\q)+b+\ell(\br)}  \Choose{|\q +\mathbb{1}|}{\q +\mathbb{1}}
 \prod_{i=1}^{\ell(\br|)} \Choose{|\br_i +\mathbb{1}|}{\br_i +\mathbb{1}} \cdot
\\
& & \hspace{20mm} \prod_{l=1}^{\ell(\q)} \left( \sum_{\p'_l \in SP([\q:\br]_l)} (-1)^{\ell(\p'_l)} \prod_{j=1}^{\ell(\p'_l)} (\ell({\p'_l}^j)-1)! \right)
\end{eqnarray*}

Note that for a partition $\T$ of a set $X$ there are exactly $\prod_{j=1}^{\ell(\T)} (\ell(\T_j)-1)!$ permutations $\sigma$ of $X$ such that the partition of $X$ obtained from the cycle decomposition of $\sigma$ is equal to $\T$.  Hence we have
$$\sum_{\T \text{ a partition of } X} (-1)^{\ell(\T)} \prod_{j=1}^{\ell(\T)} (\ell(\T_j)-1)!= \sum_{\sigma \in S_{|X|}} (-1)^{|\sigma|}.$$

Therefore if $\br\neq \q$ then the term $$\prod_{l=1}^{|\q|} \left( \sum_{\p_l \text{ is a partition of } s(\br_l) } (-1)^{|\p_l|} \prod_{j=1}^{|\p_l|} (|\p_l^j|-1)! \right)$$
vanishes, and if $\br=\q$ we get $(-1)^{\ell(\q)}$. So
\begin{eqnarray*}
\sum_{\p}  N_{\p}\K_{s(\p)} & = &  \sum_{\q }  (-1)^{\ell(\q) +b} \Choose{|\q +\mathbb{1}|}{\q +\mathbb{1}}
  \prod_{i=1}^{|\q|} \Choose{|\q_i +\mathbb{1}|}{\q_i +\mathbb{1}} 
\\
\text{(By Theorem \ref{combthm})}  & = & \left[\sum_{k=1}^{b} (-1)^{k+1} \Choose{n-1}{k-1} \right] \Choose{|B+\mathbb{1}|}{B+\mathbb{1}} (-1)^{|B|}
\\
& = & 0 .
\end{eqnarray*}
\end{proof}

\begin{theorem}\label{main-genus-zero}
Let $\A=\Set{a_1,\cdots,a_k}$ be a multi-set of integers such that $d=n-|\A|-2$. In $A^{*}(\Mbar{0}{n})$ we have:
$$\kappa_{a_1,\cdots,a_l}=\sum_{\p\in SP(\A,d)} x_{\p} \kappa_{\p(\A)}$$
where 
$$x_{\p}=\sum_{\T \leq \br \leq \p }  \frac{(-1)^{\ell(\A)+\ell(\T)+\ell(\br)}}{(|\T|+\mathbb{1})!} \prod_{j=1}^{\ell(\br)} (|\br_j+\mathbb{1}|)! \prod_{l=1}^{\ell(\p)}(\ell(\br|_{\p_l})-1)! \cdot
 (-1)^M \Choose{\ell(\T)-\ell(\br)}{M-\ell(\br)}$$ and $M=\min \{ \ell(\T),d \}$.
\end{theorem}

\begin{proof}

We define $C_k(\A) $ as follows.  
\begin{eqnarray*}
C_k(\A) & = & \sum_{ \ell(\q) =k } \K_{\q} N_{|\q|}
\\
&=& \sum_{\substack{\q \in SP(\A;k)  \\ \br \in SP(|\q|)}} \left( \prod_{i=1}^{k} \K_{\q_i} \right)  (-1)^{k+\ell(\br)} \left( \prod_{j=1}^{\ell(\br)}\Choose{|\br_j+\mathbb{1}|}{\br_j+\mathbb{1}} \right) (\ell(\br)-1)! 
\end{eqnarray*}

The data of $\q \in SP(\A;k) \text{ and } \br \in SP(|\q|)$ is equivalent to the data of $\q \leq \br \in SP(\A)$ with $\q \in SP(\A;k)$, and this is equivalent to the data of $\br \in SP(\A)$ plus the following. For each $1 \leq j \leq \ell(\br)$ a partition 
$$\q'_j:=\q|_{\br_j}=\coprod_{i=1}^{\ell(\q'_j)}{\q'_j}^i$$
 of the multi-set $\br_j$, such that $\sum \ell(\q'_j)=k$. For such pairs $(\q,\br)$ and $(\br,\Set{\q'_j}_j)$, we have
 $$\Choose{|\br_j+\mathbb{1}|}{\br_j+\mathbb{1}}=\Choose{|\br_j|+d_j}{|\q'_j|+\mathbb{1}} ,$$
  and as in the proof of Lemma \ref{NK-lemma} we have $\prod_i \K_{\q_i}=\prod_{i,j} \K_{{\q'_j}^i}$.

\begin{eqnarray*}
C_k(\A) & = & 
\sum_{\substack{\br \in SP(\A) \\ (d_1,\cdots,d_{\ell(\br)}) \\  \sum d_j=k}}
(-1)^{k+\ell(\br)}  (\ell(\br)-1)!
\prod_{j=1}^{\ell(\br)}  \left( \sum_{\q'_j \in SP(\br_j;d_j) } \Choose{|\br_j|+d_j}{|\q'_j|+\mathbb{1}}
\prod_{i=1}^{d_j} \K_{{\q'_j}^i} \right)
\\
& = &  
\sum_{\substack{\br \in SP(\A) \\ (d_1,\cdots,d_{\ell(\br)}) \\  \sum d_j=k}}
(-1)^{k+\ell(\br)}  (\ell(\br)-1)! \cdot
\\
& & \hspace{10mm}
\prod_{j=1}^{\ell(\br)}  \left( \sum_{\substack{\q'_j \in SP(\br_j;d_j) \\ \T_j^i\in SP({\q'_j}^i) }} \Choose{|\br_j|+d_j}{|\q'_j|+\mathbb{1}}
 \left( \prod_{i=1}^{d_j} (-1)^{\ell({\q'_j}^i)+\ell(t_j^i)} \Choose{||\T_j^i|+\mathbb{1}|}{|\T_j^i|+\mathbb{1}}   \right) \right)
\\
 & = &  
\sum_{\substack{\br \in SP(\A) \\ (d_1,\cdots,d_{\ell(\br)}) \\  \sum d_j=k}}
(-1)^{k+\ell(\br)}  (\ell(\br)-1)! \cdot
\\
& & \hspace{10mm}
\prod_{j=1}^{\ell(\br)}  \left( \sum_{\substack{\T_j \in SP(\br_j) \\ \tq_j\in SP(|\T_j|;d_j) }} (-1)^{\ell(\br_j)+\ell(t_j)}  \Choose{|\br_j|+d_j}{|\tq_j|+\mathbb{1}}
 \left( \prod_{i=1}^{d_j} \Choose{||\tq_j^i|+\mathbb{1}|}{|\tq_j^i|+\mathbb{1}}   \right) \right)
\\
\end{eqnarray*}

In the second line we expanded $\K_{{\q'_j}^i}$, and in the third line we switch the role of ${\q'_j}^i$ and $\T_j^i$ which by now is standard. In the fourth line we used Theorem \ref{combthm}. 

Now we switch the role of $\T$ and $\br$, and take the sum over $\T$ and $\tbr$. We use the fact that $||\T_i|+\mathbb{1}|$ for a partition $\T_i\in SP(\br_i)$ is equal to $|\tbr_i+\mathbb{1}|$ for the corresponding component of $\tbr\in SP(\T)$ (as we explained in the proof of lemma \ref{NK-lemma}), and also $(|\T|+\mathbb{1})!=\prod (|\T_i|+\mathbb{1})!$.

\begin{eqnarray*}
C_k(\A)  
& = &  
\sum_{\substack{\br \in SP(\A) \\ (d_1,\cdots,d_{\ell(\br)}) \\  \sum d_j=k}}
(-1)^{k+\ell(\br)}  (\ell(\br)-1)! \cdot
\\
& & \hspace{10mm}
\prod_{j=1}^{\ell(\br)}  \left( \sum_{\T_j \in SP(\br_j) } (-1)^{\ell(\br_j)+\ell(t_j)}  \Choose{||\T_j|+\mathbb{1}|}{|\T_j|+\mathbb{1}}
\Choose{\ell(\T_j)-1}{d_j-1} \right)
\\
& = &  \sum_{\T \in SP(\A)} \frac{(-1)^{\ell(\A)+k+|\T|}} {(|\T|+\mathbb{1})!} \cdot 
\\
& & \left( \sum_{\tbr \in SP(|\T|)}  (-1)^{\ell(\tbr)} (\ell(\tbr)-1)! \left( \prod_{j=1}^{\ell(\tbr)} (|\tbr_j+\mathbb{1}|)! \right) \left[ \sum_{\sum d_j=k} \prod_{j=1}^{\ell(\tbr)} \Choose{\ell(\tbr_j)-1}{d_j-1} \right] \right)
\\
& = &  \sum_{\T \in SP(\A)} \frac{(-1)^{\ell(\A)+k+|\T|}} {(|\T|+\mathbb{1})!} \cdot
\\
& & \left( \sum_{\tbr \in SP(|\T|)}  (-1)^{\ell(\tbr)} (\ell(\tbr)-1)! \left( \prod_{j=1}^{\ell(\tbr)} (|\tbr_j+\mathbb{1}|)! \right) \Choose{\ell(\T)-\ell(\tbr)}{k-\ell(\tbr)} \right)
\end{eqnarray*}
Therefore:
\begin{eqnarray*}
x_{\p} & = & \sum_{\substack{\q \leq \p \\ \ell(\q)\leq d} } \K_{q} N_{[\p:\q]} 
\\
& = &\sum_{\substack{(k_1,\cdots,k_{\ell(\p)}) \\ \sum_{j=1}^{\ell(\p)} k_j \leq d}} \prod_{i=1}^{\ell(\p)}\left( \sum_{\substack{\q \in SP(\p_i) \\ \ell(\q'_i)=k_i } } \K_{\q'_i} N_{|\q'_i|} \right)
\\
& = & \prod_{\sum_{i=1}^{\ell(\p)} k_i\leq d}  C_{k_i}(\p_i)
\\
& = & \prod_{\sum_{i=1}^{\ell(\p)} k_i\leq d}  \sum_{\T_i \in SP(\p_i)} \frac{(-1)^{\ell(\p_i)+k_i+|\T_i|}} {(|\T_i|+\mathbb{1})!} \cdot 
\\
& & \left( \sum_{\tbr_i \in SP(|\T_i|)}  (-1)^{\ell(\tbr_i)} (\ell(\tbr_i)-1)! \left( \prod_{j=1}^{\ell(\tbr_i)} (|\tbr_i^j+\mathbb{1}|)! \right) \Choose{\ell(\T_i)-\ell(\tbr_i)}{k_i-\ell(\tbr_i)} \right) 
\end{eqnarray*}

By reordering the terms in the expression and using Theorem \ref{combthm}, we get the required result.

\begin{eqnarray*}
x_{\p}
& = &  \sum_{{\substack{\T \leq \br \leq \p \\ k\leq d} }}  \frac{(-1)^{\ell(\A)+k+\ell(\T)+\ell(\br)}}{(|\T|+\mathbb{1})!} \prod_{j=1}^{\ell(\br)} (|\br_j+\mathbb{1}|)! \prod_{l=1}^{\ell(\p)}(\ell(\br|_{\p_l})-1)! \cdot
\\
& & \hspace{20mm} \left( \sum_{\sum k_i =k} \prod_{i=1}^{\ell(\p)} \Choose{\ell(\T|_{\p_i})-\ell(\br|_{\p_l})}{ k_i-\ell(\br|_{\p_l})} \right)
\\
& = &  \sum_{\T \leq \br \leq \p }  \frac{(-1)^{\ell(\A)+\ell(\T)+\ell(\br)}}{(|\T|+\mathbb{1})!} \prod_{j=1}^{\ell(\br)} (|\br_j+\mathbb{1}|)! \prod_{l=1}^{\ell(\p)}(\ell(\br|_{\p_l})-1)! \cdot
\\
& & \hspace{20mm} \left( \sum_{\sum k =\ell(\br)}^{d} (-1)^k \Choose{\ell(\T)-\ell(\br)}{k-\ell(\br)}  \right)
\\
& = &  \sum_{\T \leq \br \leq \p }  \frac{(-1)^{\ell(\A)+\ell(\T)+\ell(\br)}}{(|\T|+\mathbb{1})!} \prod_{j=1}^{\ell(\br)} (|\br_j+\mathbb{1}|)! \prod_{l=1}^{\ell(\p)}(\ell(\br|_{\p_l})-1)! \cdot
 (-1)^M \Choose{\ell(\T)-\ell(\br)}{M-\ell(\br)} 
\\
& & \text{where }M=\min \{ \ell(\T),d \}.
\end{eqnarray*}

\end{proof}

\begin{proof}[Proof of Theorem \ref{main}]
As we mentioned in the introduction, Theorem \ref{iota} implies that the expression of Theorem \ref{main-genus-zero} hold in $\kappa^*(\mathcal{M}_{g,n}^{ct})$ for any genus $g$. This completes the proof of Theorem \ref{main}. 

\end{proof}

\begin{example}
We consider two special cases of Theorem \ref{main}. In both cases, in order to get a non-zero contribution from $ \Choose{\ell(\T)-\ell(\br)}{M-\ell(\br)}$, we have to consider special $\br$ and $\T$.

\begin{itemize}
\item $d=1$ so $M=1$. In order to get a non-zero contribution we must have $\ell(\br)=1$. There we obtain:
$$x_{\p}=\sum_{\T  \in SP(\A)}  \frac{(-1)^{\ell(\A)+\ell(\T)}}{(|\T|+\mathbb{1})!} (|\A+\mathbb{1}|)!  ,$$ that is the result of Proposition \ref{baby-case}.
\item $d=n$ which implies $M=\ell(\T)$. In order to get a non-zero contribution we must have $\br=\T$. So 
$$x_{\p}=\sum_{\br \leq  \p \in SP(\A)}  (-1)^{\ell(\A)+\ell(\br)} \prod_{l=1}^{\ell(\p)}(\ell(\br|_{\p_l})-1)! , $$ which vanishes unless $\ell(\p_i)=1$ for each $i$. 
\end{itemize}
\end{example}

\section{Combinatorial identity}
This section contains the statement and proof of the combinatorial identity (Theorem \ref{combthm}) that we used in the previous section. The methods are completely combinatorial and are independent from the rest of the paper. 

\begin{lemma}\label{comblemma}
Let $\A=\Set{a_1,\cdots,a_n}$ be a multi-set of integers. 
$$\sum_{\p \in SP(\A;k) } \prod_{i=1}^k |\p_i|^{\ell(\p_i)-1}= {\Choose{n-1}{k-1}} (a_1+\cdots+a_n)^{n-k} .$$
\end{lemma}
\begin{proof}
Let $K_{n+1}$ be the complete graph on $n+1$ vertices with vertices labeled by $1=a_0,a_1,\cdots,a_n$. For a spanning tree $T\subset K_{n+1}$ we denote the Pr\"{u}fer code of $T$ by $P(T)$, and 
$$p(T):=\text{ the product of all the entries of }P(T)$$
 which we call the Pr\"{u}fer function of $T$. For definition and basic properties of the Pr\"{u}fer code see \cite{Vanlint} page 13. 

Let $\Tk$ be the set of spanning tree of $K_{n+1}$ with $deg(v_0)=k$. We compute the sum 
$S=\sum_{T \in \Tk } p(T)$ in two ways. 

\begin{enumerate}
\item If we forget the condition $deg(v_0)=k$, then any sequence of length $n-1$ of $a_i$'s would appear exactly once as the Pr\"{u}fer code of a spanning tree of $K_{n+1}$. For trees in $\Tk$ the label $a_0=1$ appears exactly $k-1$ times, and the remaining entries are filled with the rest of $a_i$'s with no extra conditions. 
In order to specify the code for a tree $T\in T_k$, we have to chose the location of the $k-1$ $a_0-$entries, where we have ${\Choose{n-1}{k-1}}$ choices, and the remaining $n-k$ entries are filled arbitrary with the remaining $a_i$'s, which contributes $(a_1+\cdots+a_n)^{n-k}$ to $S$. Therefore $$S={\Choose{n-1}{k-1}} (a_1+\cdots+a_n)^{n-k} .$$
\item For a tree $T\in \Tk$ the vertex $v_0$ is connected to exactly $k$ disjoint sub-trees $\tilde{G}_1,\dots,\tilde{G}_k$ of $T$. $v_0$ is connected to a unique vertex in each  $\tilde{G}_i$, and we denote by $G_i$ the union of $\tilde{G}_i$, $v_0$ and the edge connecting them. The vertex set of $\tilde{G}_i$'s gives a partition $\p$ of $\A$ into exactly $k$ parts. The Pr\"{u}fer code of each $G_i$ uniquely determines $G_i$, and if we know all the $G_i$'s then $T$ and the Pr\"{u}fer code of it are uniquely determined. Therefore we have: 
\begin{eqnarray*}
S & = &  \sum_{T \in \Tk } p(T) \\
 & = & \sum_{\p \in SP(\A;k)} \prod_{i=1}^k \left(\sum_{\substack{  G_i :\text{ spanning tree on} \\ \text{ the vertex set } \p_i\cup v_0 \\ deg_{G_i}(v_0)=1}} p(G_i) \right)
\end{eqnarray*}
Note that the Pr\"{u}fer code of each $G_i$ is a sequence of length $\ell(\p_i)-1$, and the entries are from $\p_i$. If we take the sum over all such trees any sequence of length $\ell(\p_i)-1$ with entries in $\p_i$ appears exactly once, so $$\sum_{\substack{  G_i :\text{ spanning tree on} \\ \text{ the vertex set } \p_i\cup v_0 \\ deg_{G_i}(v_0)=1}} p(G_i)=|\p_i|^{\ell(\p_i)-1} .$$ Thus we have $$S=\sum_{\p \in SP(\A;k) } \prod_{i=1}^k |\p_i|^{\ell(\p_i)-1} .$$

\end{enumerate}

\end{proof}

\begin{definition}

For $x\in \mathbb{R}$ we denote the falling factorial of $x$ by $$(x)_n:=x(x-1)\cdots (x-(n-1)) .$$
Let $P(x_1,\cdots,x_n)\in \mathbb{R}[x_1,\cdots,x_n]$ be a polynomial in $n$ variables. If $$P=\sum c_{i_1,\cdots,i_n} x_1^{i_1} \cdots x_n^{i_n} ,$$ we define the \emph{falling factorialization} of $P$ to be $$FF(P)=\sum c_{i_1,\cdots,i_n} (x_1)_{i_1} \cdots (x_n)_{i_n} .$$
\end{definition}

\begin{remark}\label{FFmultinomial}
It is straight forward to see that we have:
$$(x+y)_n=\sum_{i=0}^n \Choose{n}{i} (x)_i (y)_{n-i} .$$

Using induction on the number of variables, we see that for falling factorials we have the multinomial coefficient theorem.
$$(x_1+\cdots+x_n)_k=\sum_{k_1+\cdots+k_n=k}\Choose{k}{k_1,\cdots,k_n} (x_1)_{k_1} \cdots (x_n)_{k_n}$$

\end{remark}

\begin{theorem}\label{combthm}
Let $\A=\Set{a_1,\cdots,a_n}$ be a multi-set of integers. 
$$\sum_{\p\in SP(\A;k)}{\Choose{||\p|+\mathbb{1}|}{|\p|+\mathbb{1}}} \prod_{i=1}^k {\Choose{|\p_i+\mathbb{1}|}{\p_i+\mathbb{1}}} = {\Choose{n-1}{k-1}} {\Choose{|\A+\mathbb{1}|}{\A+\mathbb{1}}} .$$
\end{theorem}
\begin{proof}
We can rewrite the identity as
$$\sum_{\p\in SP(\A;k)}  (|\p_i+\mathbb{1}|)_{\ell(\p_i)-1} =  {\Choose{n-1}{k-1}} (|\A+\mathbb{1}|)_{n-k} ,$$
or equivalently, if we set $\B=\A+\mathbb{1}$ as
\begin{equation}\label{FFidentity}
\sum_{\p\in SP(\B;k)} \prod_{i=1}^k (|\p_i|)_{\ell(\p_i)-1} =  {\Choose{n-1}{k-1}} (|\B|)_{n-k} .
\end{equation}

Note that by Remark \ref{FFmultinomial} the left hand side of (\ref{FFidentity}) is the falling factorialization of $$\sum_{\p \in SP(\B;k) } \prod_{i=1}^k |\p_i|^{\ell(\p_i)-1} ,$$ and similarly the right hand is the falling factorialization of $${\Choose{n-1}{k-1}} (|\B|)^{n-k} .$$

By Lemma \ref{comblemma} $$\sum_{\p \in SP(\B;k) } \prod_{i=1}^k |\p_i|^{\ell(\p_i)-1} = {\Choose{n-1}{k-1}} (|\B|)^{n-k} ,$$
so their falling factorializations are also equal, which completes the proof.
\end{proof}



\bibliography{}

\end{document}